\documentclass[a4paper,12pt]{article}
\usepackage[cp1251]{inputenc}
\usepackage[russian, ukrainian,english]{babel}

\usepackage{amsmath, amsthm, amsfonts, amssymb}
\usepackage{graphicx}

\theoremstyle{plain}
\newtheorem{theorem}{Theorem}[section]
\newtheorem{example}{Example}[section]
\newtheorem{lemma}{Lemma}[section]
\newtheorem{corollary}{Corollary}[section]
\newtheorem{remark}{Remark}[section]
\newtheorem{definition}{Definition}[section]

\begin{document}
	\begin{center}
	{\Large\bf Random dynamical systems generated by coalescing stochastic flows on $\mathbb{R}$}
	
\vskip10pt

G. V. Riabov

\vskip10pt

{\it Institute of Mathematics, NAS of Ukraine, Kyiv}
		\end{center}
		
	\vskip30pt

{\bf Abstract.} Existence of random dynamical systems for a class of coalescing stochastic flows on $\mathbb{R}$ is proved. A new state space for coalescing flows is built. As particular cases coalescing flows of solutions to stochastic differential equations and coalescing Harris flows are considered.

\section{Introduction}	

In the present paper we prove the existence of a random dynamical system generated by a coalescing stochastic flow on $\mathbb{R}.$ As an introductory example consider the Arratia flow -- a typical and one of the most studied coalescing stochastic flows. The Arratia flow is a family of Wiener processes $\{w_{(s,x)}(t):s\leq t, x\in \mathbb{R}\}$ (with respect to the joint filtration) that start from every time-space point of the plane $(s,x)\in \mathbb{R}\times \mathbb{R},$ $w_{(s,x)}(s)=x.$ Motion of any finite system of processes $(w_{(s,x_1)},\ldots,w_{(s,x_n)})$ up to the first meeting time 
$$
\sigma_n=\inf\{t\geq s: \ w_{(s,x_i)}(t)=w_{(s,x_j)}(t) \mbox{ for some } i\ne j \}
$$
coincides with the motion of $n$ independent Wiener processes up to the first meeting time, and at the moment of meeting processes coalesce:
$$
w_{(s_1,x_1)}(t)=w_{(s_2,x_2)}(t) \Rightarrow w_{(s_1,x_1)}(t')=w_{(s_2,x_2)}(t') \mbox{ for all } t'\geq t.
$$
The Arratia flow originated from \cite{Arratia, Arratia2} as a scaling limit in the voter model on $\mathbb{Z}$. Properties of the flow with different applications were studied in \cite{LJR, Darling, Harris, FINR, SSS, D1, TW}. In a recent paper \cite{NT} the Arratia flow was proved to be a limiting object in the Hastings-Levitov diffusion-limited aggregation model (corresponding to the rate $\alpha=0$).

We study random mappings of $\mathbb{R}$ generated by a coalescing stochastic flow. Even for the Arratia flow their structure turned out to be rather subtle. As far as we aware there still was no construction of a measurable (in all variables) dynamical version of the Arratia flow:
$$
w_{s,x}(t,\omega)=w_{t-s,x}(t,\theta_s \omega) \mbox{ for all } s\leq t, x\in \mathbb{R}, \omega\in \Omega,
$$
where $\{\theta_s,s\in\mathbb{R}\}$ is (at least) jointly measurable group of measure preserving transformations of the underlying probability space. 
Basic problem here is the discontinuity of mappings $x\to w_{(s,x)}(t)$  (see fig. 1).

\begin{figure}[!h]
 \caption{Mappings from the Arratia flow are step functions in $x$.}
 \centering
\includegraphics[width=9cm]{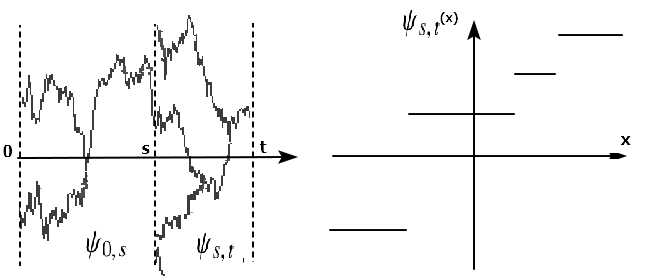}
\end{figure}

For comparison we mention the case of smooth (in $s,x,t$) stochastic flows. Such flows were extensively studied in the context of stochastic differential equations (see \cite[Ch. 4]{Kunita} and references therein). Consider a stochastic differential equation 
\begin{equation}
\label{eq28_2}
dX(t)=a(X(t))dt+b(X(t))dw(t)
\end{equation}
with bounded Lipschitz coefficients $a,b:\mathbb{R}\to\mathbb{R}.$ It is possible to define solutions to \eqref{eq28_2} simultaneously for all initial conditions $(s,x)\in\mathbb{R}\times \mathbb{R}.$ More precisely, there exists a continuous stochastic process  $\{\psi_{s,t}(x):-\infty<s\leq t<\infty,x\in \mathbb{R}\}$ such that for all $(s,x,t)$ one has
\begin{equation}
\label{eq28_1}
\psi_{s,t}(x)=x+\int^t_sa(\psi_{s,r}(x))dr+\int^t_s b(\psi_{s,r}(x))dw(r),
\end{equation}
and random mappings $x\to \psi_{s,t}(x)$ are related by the flow property:
\begin{equation}
\label{eq18_4_2}
\psi_{s,t}\circ \psi_{r,s}=\psi_{r,t}, \ r\leq s\leq t.
\end{equation}
The structure of mappings $\{\psi_{s,t}\}$ can be further specified. Let the underlying probability space be the classical Wiener space: $\Omega$ is the space $C_0(\mathbb{R},\mathbb{R})$ of all con\-ti\-nu\-o\-us functions $f:\mathbb{R}\to\mathbb{R},$ $f(0)=0,$ equipped with the metric of uniform convergence on compacts, $\mathcal{F}$ is the Borel $\sigma-$field on $C_0(\mathbb{R},\mathbb{R}),$ and $\mathbb{P}$ is the (two-sided) Wiener measure. Then the canonical process $w(f,t)=f(t)$
is the (two-sided) Wiener process. Consider a group of shifts $(\theta_h)_{h\in\mathbb{R}},$ $\theta_hf(t)=f(t+h)-f(h),$
which is a continuous group of measure-preserving transformations. By \cite[Th. 2.3.40]{Arnold} solutions to \eqref{eq28_1} can be organized in such a way that the equality
\begin{equation}
\label{eq18_4_1}
\psi_{s,t}(f,x)=\psi_{0,t-s}(\theta_s f,x)
\end{equation}
holds for all $s,x,t,f$ (with no exceptions in $f$). Thus $\varphi(t,f,x)=\psi_{0,t}(f,x)$ is a random dynamical system that represents the flow of solutions to \eqref{eq28_1}. Its existence allows to apply results from the theory of dynamical systems (e.g. ergodic theorems, theorems on attractors and stability, transformations of measures by stochastic flows) in the study of stochastic flows. A number of such applications is given in \cite{Arnold, Kunita}. A typical approach to the construction of a modification \eqref{eq18_4_1} is to build a crude cocycle (so that \eqref{eq18_4_1} holds for $\mathbb{P}-$a.a. $f\in \Omega$) of regular enough random mappings $\psi_{s,t}$ (e.g. elements of some Polish group).  Then general perfection theorems  \cite{Arnold, Arnold_Scheutzow} allow to choose a modification such that \eqref{eq18_4_1} holds without exceptions. Described approach cannot be used for stochastic flows with coalescence due to the absense of a good functional space where random mappings  $x\to\psi_{s,t}(x)$ live. To discuss possible approaches and results in the non-smooth case we give rigorous definitions of a stochastic flow and a random dynamical system.

We describe stochastic flows in terms of $n-$point motions in the manner of \cite{LJR}. Consider a sequence of transition probabilities $\{P^{(n)}:n\geq 1\}$ where $\{P^{(n)}_t:t\geq 0\}$ is a transition probability on $\mathbb{R}^n.$ Assume that transition probabilities satisfy following conditions.

{\bf TP1} (Feller condition) For each $n\geq 1$ the expression 
$$
T^{(n)}_t f(x)=\int_{\mathbb{R}^n} f(y)P^{(n)}_t(x,dy), \ x\in \mathbb{R}^n
$$
defines a Feller semigroup on $C_0(\mathbb{R}^n)$ \cite[Ch.4, \S 2]{EK}.

{\bf TP2} (consistency) Given $1\leq i_1< i_2<\ldots <i_k\leq n,$ $B_k\in \mathcal{B}(\mathbb{R}^k)$ and $C_n=\{y\in \mathbb{R}^n: (y_{i_1},\ldots,y_{i_k})\in B_k\}$  one has
$$
P^{(n)}_t(x,C_n)=P^{(k)}_t((x_{i_1},\ldots,x_{i_k}),B_k), \ t\geq 0, x\in \mathbb{R}^n.
$$

{\bf TP3} (coalescing condition) For all $x\in \mathbb{R}$ 
\begin{equation}
\label{eq24_2}
P^{(2)}_t((x,x),\Delta)=1, \ t\geq 0,
\end{equation}
where $\Delta=\{(y,y):y\in \mathbb{R}\}$ is the diagonal in the space $\mathbb{R}^2.$

{\bf TP4} (continuity of trajectories) For all $x\in\mathbb{R}$ and $\varepsilon>0$ one has
$$
t^{-1}P^{(1)}_t(x,(x-\varepsilon,x+\varepsilon)^c)\to 0, \ t\to 0.
$$
\noindent
Under this condition a transition probability $P^{(1)}$ generates a continuous Feller process on $\mathbb{R}$ \cite[Ch. 4, Prop. 2.9]{EK}.  Conditions  {\bf TP1, TP2} imply that for each $n\geq 1$ there exists a family $\{\mathbb{P}^{(n)}_x,x\in\mathbb{R}^n\}$ of probability meaures on  $C([0,\infty),\mathbb{R}^n)$ such that with respect to $\mathbb{P}^{(n)}_x$ the canonical process $X^{(n)}_t(f)=f(t),$ $f\in C([0,\infty),\mathbb{R}^n),$ is a continuous Markov process with a transitional probability $\{P^{(n)}_t:t\geq 0\}$ and a starting point $x$  \cite[Ch. 4, Th. 1.1]{EK}.

{\bf TP5} (local estimates on a meeting time) For each $c<c'$ and $t>0$ there exists a continuous increasing function $m:\mathbb{R}\to\mathbb{R}$ such that for all $x,y$
$$
\begin{aligned}
\mathbb{P}^{(2)}_{(x,y)} \Big(\forall s\in[0,t] \ \  & \big(X^{(2)}_1(s),X^{(2)}_2(s)\big)\in [c,c']^2  \\
&   \mbox{ and } X^{(2)}_1(s)\ne X^{(2)}_2(s)\Big)\leq |m(x)-m(y)|.
\end{aligned}
$$

{\bf TP6} (absence of atoms) For all $t> 0$ and $x\in\mathbb{R}$ the measure $P^{(1)}_t(x,\cdot)$ has no atoms.

By \cite[Th. 1.1]{LJR} conditions {\bf TP1-TP3} are enough for the construction of a (weak) stochastic flow of random mappings of $\mathbb{R}$ with finite-point motions defined by transition probabilities $\{P^{(n)}:n\geq 1\}$ in the sense of the following definition.

\begin{definition}\label{def29_1} \cite[Def. 1.6]{LJR} A stochastic flow of mappings of $\mathbb{R}$ generated by transition probabilities $\{P^{(n)}:n\geq 1\}$  is the family $\{\psi_{s,t}(x):-\infty<s\leq t<\infty,x\in \mathbb{R}\}$ of random variables that satisfy following properties.
	
	{\bf SF1} (regularity) For all $s\leq t,$ 
	$$
	\psi_{s,t}:\Omega\times \mathbb{R}\to \mathbb{R}
	$$
	is a measurable mapping.
	
	{\bf SF2} (weak flow/crude cocycle property) For all $r\leq s\leq t,$  $\omega\in \Omega$ and $x \in \mathbb{R}$
	$$
	\mathbb{P}(\psi_{s,t}(\psi_{r,s}(x))=\psi_{r,t}(x))=1.
	$$
	
	{\bf SF3} (independent and stationary increments) Denote by $\mathcal{F}^\psi_{s}$ the ``past'' $\sigma-$field, i.e. $	\mathcal{F}^\psi_{s}=\sigma(\{\psi_{p,q}(x):p\leq q\leq s, x\in \mathbb{R}\}).$ Then for any $s\leq t,$ $\mathcal{F}^\psi_{s}-$measurable $\mathbb{R}^n-$valued random vector  $\xi$  and a Borel set $B\in \mathcal{B}(\mathbb{R}^n)$ one has 
	$$
	\mathbb{P}\Big(\psi_{s,t}(\xi)\in B|\mathcal{F}^\psi_s\Big)=P^{(n)}_{t-s}(\xi,B) \ \mbox{a.s.}
	$$
	($\psi_{s,t}(\xi)$ is an abbreviation for $(\psi_{s,t}(\xi_1),\ldots,\psi_{s,t}(\xi_n))$).
	
\end{definition}

\begin{remark} Definition \ref{def29_1} differs from \cite[Def. 1.6]{LJR} mainly in the condition {\bf SF3}. We comment on the difference in the appendix and prove that within the definition \ref{def29_1} finite-dimensional distributions of the stochastic flow are uniquely determined by transition probabilities $\{P^{(n)}:n\geq 1\}.$
\end{remark}

Next we give a rigorous definition of what is meant by a random dynamical system.

\begin{definition}  \label{def12} \cite[Def. 1.1.1]{Arnold} A metric dynamical system is a quadruple 
	$$
	(\Omega,\mathcal{F},\mathbb{P};\{\theta_h,h\in\mathbb{R}\}),
	$$
	where $(\Omega,\mathcal{F},\mathbb{P})$ is a probability space and $\theta$ is a measurable group of measure-preserving transformations, i.e. the function $\theta:\mathbb{R}\times\Omega\to\Omega$ is measurable, $\theta_0\omega=\omega,$ $\theta_{h+s}=\theta_h\circ \theta_s$ (for all $\omega\in\Omega$ and $h,s\in\mathbb{R}$) and $\mathbb{P}\circ \theta^{-1}_h=\mathbb{P}.$ 
	
	A random dynamical system on $\mathbb{R}$ over $\theta$ is a mapping 
	$$
	\varphi:\mathbb{R}_+\times\Omega\times \mathbb{R}\to \mathbb{R}
	$$
	that satisfies two properties.
	
	{\bf RDS1} (measurability) The function $\varphi:\mathbb{R}_+\times\Omega\times \mathbb{R}\to \mathbb{R}$ 	is (jointly) measurable.
	
	{\bf RDS2} (perfect cocycle property)  For all $\omega\in \Omega,$  $s,t\geq 0$ and $x\in \mathbb{R},$ 
	$$
	\varphi(0,\omega,x)=x \mbox{ and }	\varphi(t+s,\omega,x)=\varphi(t,\theta_s\omega,\varphi(s,\omega,x)).
	$$
\end{definition}
We will say that a stochastic flow of mappings $\{\psi_{s,t}:-\infty<s\leq t<\infty\}$ is generated by a random dynamical system $\varphi,$ if for all $s\leq t,$ $\omega\in \Omega$ and $x\in \mathbb{R}$
\begin{equation}
\label{eq_rds}
\psi_{s,t}(\omega,x)=\varphi(t-s,\theta_s\omega,x).
\end{equation}

The general result of \cite{LJR} does not give the full property {\bf RDS2} -- it only gives a very crude cocycle, i.e. the equality \eqref{eq_rds} holds outside a set of measure zero (that depends on $s,t,x$). A measurable modification of the Arratia flow with {\bf SF2} holding without exceptions was built already in \cite{Arratia2}. But that construction is not consistent with time shifts, e.g. right-continuity of mappings $x\to \psi_{s,t}(x)$ is intended only for rational $s.$ Respectively, measure preserving transformations $\theta_h$ are absent in \cite{Arratia2}. Another version of the Arratia flow, which is a perfect cocycle (as a particular case of a very general setting) was built in \cite{Darling}. However, condition {\bf RDS1} was violated in that construction: an underlying probability space in \cite{Darling}  is a space of all families $(\psi_{s,t})_{s\leq t}$ of mappings satisfying \eqref{eq18_4_2} with $\theta$ being a group of shifts,
$$
(\theta_h\psi)_{s,t}=\psi_{s+h,t+h}.
$$
Such probability space is very large, e.g. it contains subspaces isomorphic to the space of all functions on $\mathbb{R}$ with a cylindrical $\sigma$-field (as any function $f:\mathbb{R}\to\mathbb{R}$ can be included as an element $f=\psi_{0,1}$ into the flow $(\psi_{s,t})_{s\leq t}$). Besides, lack of measurability results in a very complicated description of a resulting stochastic flow: conditional probability in {\bf SF3} is not well-defined for non-measurable mapping $\psi_{s,t}.$ Also it can be cheked that the shift $h\to \psi_{s+h,t+h}(x)$ is not measurable. One approach for the construction of a good (i.e. Polish) phase space for the Arratia flow was suggested in \cite{FINR}. It was proved that the Arratia flow uniquely defines a random compact in the space of compact sets of continuous trajectories. This random compact was called a Brownian Web (we refer to  \cite{SSS, BGS} for a detailed account on developments related to this very interesting and important object). It was pointed already in \cite{FINR} that the Brownian Web cannot be viewed as a family of random mappings -- with probability 1 there are multiple trajectories in the Brownian Web that start at the same point. Thus the problem of defining random mappings remains. Another state space for (part of) the Arratia flow $\{\psi_{0,t}(x):t\geq 0, x\in\mathbb{R}\}$ was suggested in \cite{D1, D2, DO} without studying the possibility to build a perfect cocycle of random mappings. Some results on the structure of a filtration generated by the Arratia flow were obtained in \cite{Riabov}. An interesting fact was pointed already in \cite{Arratia2}: properties {\bf RDS1, RDS2} and right-continuity of mappings $x\to \psi_{s,t}(x)$ are merely inconsistent for the Arratia flow. Described results indicate that in order to build a random dynamical system for the Arratia flow an appropriate underlying probability space must be chosen. We build such probability space in Sections 2 and 3. As a result we are able to prove existence of random dynamical systems for a large class of transition probabilities that define coalescing stochastic flows. The following theorem is the main result of the paper. 

\begin{theorem}
	\label{thm1} 	Consider a sequence  of transition probabilities $\{P^{(n)}:n\geq 1\}$  that satisfy conditions {\bf TP1-TP6} above. Then for a suitable metric dynamical system $(\Omega,\mathcal{F},\mathbb{P}, \{\theta_h,h\in \mathbb{R}\})$ there exists a random dynamical system 
	$$
	\varphi: \mathbb{R}_+\times\Omega\times \mathbb{R}\to \mathbb{R}
	$$
	over $\theta$ such that
	$$
	\psi(s,t,\omega,x)=\varphi(t-s,\theta_s \omega,x)
	$$
	is a stochastic flow of mappings generated by transition probabilities $\{P^{(n)}:n\geq 1\}.$
	
\end{theorem}

The approach we take is inspired by \cite{Darling, FINR, D2}. The construction is realized on a new space of coalescing stochastic flows $(\mathbb{F},\mathcal{A})$. Its measurable structure together with a measurable group $\theta$ (which is a group of shifts) and a perfect cocycle $\varphi$ are described in section 2. In section 3 starting from a sequence of transition probabilities $\{P^{(n)}:n\geq 1\}$ that satisfy {\bf TP1-TP6} we build a probability measure $\mu$ on $(\mathbb{F},\mathcal{A})$ and finish the proof of existence of a random dynamical system. Finally, in section 4 our construction is applied to prove existence of random dynamical systems for two types of stochastic flows: coalescing stochastic flows of solutions of an SDE \eqref{eq28_2} that move independently before the meeting time and a class of coalescing Harris flows \cite{Harris}. These examples cover the case of the Arratia flow as well.

\section{Measurable Space of Flows}
Let $C_x([s,\infty))$ be the space of all continuous functions $g:[s,\infty)\to\mathbb{R},$ $g(0)=x.$ Respectively, $\prod_{(s,x)}C_x([s,\infty))$ is the set of all families $f=(f(s,x;\cdot))_{(s,x)\in \mathbb{R}^2}$ such that $t\to f(s,x;t)$ is a continuous function on $[s,\infty),$ $f(s)=x.$ In the following definition the probability space that will carry a random dynamical system is described. 

\begin{definition}\label{def2} The space of flows $\mathbb{F}$ is the set of all families $f\in \prod_{(s,x)}C_x([s,\infty))$  that satisfy following conditions.

	{\bf F1} For all $r\leq s\leq t,$ $x\in\mathbb{R}$
	$$
	f(s,f(r,x;s);t)=f(r,x;t);
	$$
	
	{\bf F2} For all $s\in\mathbb{R}$ the set $\mathcal{R}_s(f)=\{f(r,x;s):r<s,x\in\mathbb{R}\}$  is dense in $\mathbb{R}.$
	
	{\bf F3} For all $s\leq t$ the mapping
	$$
	x\to f(s,x;t)
	$$
	is right-continuous at each point $x\not\in\mathcal{R}_s(f).$
	
	{\bf F4} For all $s<t$ and $x\in\mathbb{R}$ there exist $r<t$ and $y\not\in\mathcal{R}_r(f)$ such that
	$$
	f(s,x;t)=f(r,y;t).
	$$

\end{definition}

We will refer to the set $\mathcal{R}_s(f)$ as to the range of the flow at time $s.$ Points  $x\not\in\mathcal{R}_s(f)$ will be called  ``fresh'' points at time $s.$  So, the range of the flow must be dense at each moment of time. On the other hand there are sufficiently many fresh points -- the condition {\bf F4} means that each trajectory of the flow immediately coalesces with a trajectory started from a fresh point.  For example, the family $f(s,x;t)=x$ does not belong to $\mathbb{F},$ because it has no fresh points.  Right-continuity is intended only at fresh points. The choice ``right-'' is made to achieve needed measurability properties (see lemma \ref{lem1} below) and it can be changed e.g. to ``left-'' without affecting the results. Let us give an example of the flow $f\in \mathbb{F}.$

\begin{figure}[!h]
	\centering
	\caption{Trajectories from the flow $f$ (example \ref{ex10_1}) for $-1\leq x< 2$.}
	\includegraphics[width=8cm]{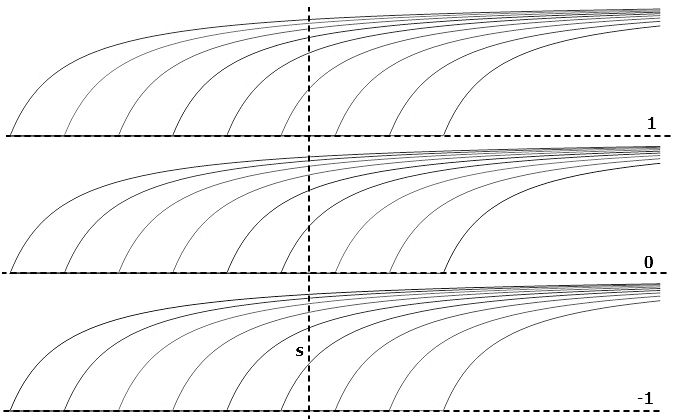}
\end{figure}

\begin{example}\label{ex10_1}
	Consider any homeomorphism  $g:[0,\infty)\to [0,1).$ Then the family
	$$
	f(s,x;t)=[x]+g(t-s+g^{-1}(x-[x]))
	$$
	belongs to $\mathbb{F}$ (fig. 2). Indeed, at every moment of time $s$ the range is $\mathbb{R}\setminus \mathbb{Z}.$ Respectively, fresh points are integers and every trajectory starts from an integer point at a certain moment of time. Right-continuity holds for all points.

\end{example}

As a corollary of the flow property {\bf F1} and continuity we have monotonicity:

{\bf F5} for all $s\leq t$
$$
x\leq y\Rightarrow f(s,x;t)\leq f(s,y;t).
$$
We equip $\mathbb{F}$ with a cylindrical $\sigma-$field $\mathcal{A},$ i.e. the smallest $\sigma-$field that makes all mappings $f\to f(s,x;t)$ measurable. The built space combines two features. Firstly, it is invariant under the shift $(\theta_h f)(s,x;t)=f(s+h,x;t+h).$ Secondly, the evaluation mapping on $\mathbb{F}$ possesses good measurability properties that are stated in the lemma \ref{lem1}. Denote by $H$ the half space $\{(s,t)\in\mathbb{R}^2:s\leq t\}.$ It is equipped with the Borel $\sigma-$field $\mathcal{B}(H).$

\begin{lemma}
	\label{lem1} The mapping $(s,t,x,f)\to f(s,x;t)$
	is $\mathcal{B}(H)\times\mathcal{B}(\mathbb{R})\times \mathcal{A}/\mathcal{B}(\mathbb{R})$--measurable.
	
\end{lemma}

\begin{proof} For fixed $(s,x)$ the mapping $(t,f)\to f(s,x;t)$ is measurable due to continuity in $t$ \cite[Ch. II, L. (73.10)]{RW}. Thus problems may occur because of variables $(s,x)$ that define a starting point of $f.$ The following relation allows to take $(s,x)$ out of the starting point and thus it is enough for the proof of measurability.
	
	\begin{equation}
	\label{eq_5_4}
	f(s,x;t)<c \Leftrightarrow \exists (p,u)\in\mathbb{Q}^2: \ p<s, f(p,u;s)\geq x, f(p,u;t)<c.
	\end{equation}
	The sufficiency immediately follows from the flow property {\bf F1} and monotonicity condition {\bf F5}: if $f(p,u;s)\geq x, f(p,u;t)<c,$ then
	$$
	f(s,x;t)\leq f(s,f(p,u;s);t)=f(p,u;t)<c.
	$$
	
	To prove necessity, consider two possibilities. Case 1:
	$$
	\exists (r,y): r\leq s, f(r,y;s)>x, f(r,y;t)<c.
	$$
	By the density property {\bf F2} there are at least two points of the range $\mathcal{R}_s(f)$ in the interval $(x,f(r,y;s)).$ In other words, there are points $q_1,q_2<s$ and $v_1,v_2$ such that
	$$
	x<f(q_1,v_1;s)<f(q_2,v_2;s)<f(r,y;s).
	$$
	
		\begin{figure}[!h]
			\centering
			\caption{There exists a trajectory that passes strictly above $(s,x)$ and below $(t,c)$.}
			\includegraphics[width=8cm]{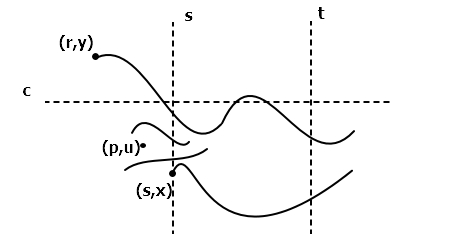}
		\end{figure}

	Consider rational numbers $p,u$ such that  $p\in (q_1\vee q_2,s)$ and $u \in (f(q_1,v_1;p),f(q_2,v_2;p)).$ By monotonicity and the flow property
	$$
	f(p,u;s)\geq f(p,f(q_1,v_1;p);s)=f(q_1,v_1;s)>x,
	$$
	$$
	f(p,u;t)\leq f(p,f(q_2,v_2;p);t)=f(q_2,v_2;t)=
	$$
	$$
	=f(s,f(q_2,v_2;s);t)\leq f(s,f(r,y;s);t)=f(r,y;t)<c.
	$$
	The point $(p,u)$ is the needed one. Case 1 is illustrated in figure 3.

	Case 2 is the complementary case:
	\begin{equation}
	\label{eq_5_3}
	r\leq s, f(r,y;s)>x \Rightarrow f(r,y;t)\geq c.
	\end{equation}
	It follows that $x$ is not a fresh point at time $s$. Indeed, if $x\not\in\mathcal{R}_s(f),$ then by the right-continuity ({\bf F3}) there exists $y>x$ such that $f(s,y;t)<c.$ Then the choice $r=s$ and $y$ contradicts \eqref{eq_5_3}. So, $x$ is not a fresh point at time $s$ and by the condition {\bf F4}
	$$
	x=f(r,y;s)
	$$
	for some $r<s$ and $y\not\in \mathcal{R}_r(f).$ 
	
	\begin{figure}[!h]
		\centering
		\caption{Point $(s,x)$ lies on a  trajectory that started from a fresh point $y$ at time $r$.}
		\includegraphics[width=8cm]{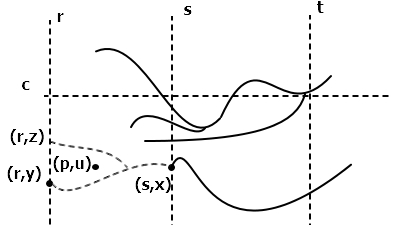}
	\end{figure}
	
	For the point $(r,y)$ we have
	$$
	f(r,y;t)=f(s,x;t)<c.
	$$
	By right-continuity ({\bf F4}) there exists a point $z>y$ such that $f(r,z;t)<c$ (see fig.  4). But $f(r,z;s)\geq x$ and condition \eqref{eq_5_3} implies that $f(r,z;s)=x.$ Choose a rational point $(p,u)\in \mathbb{Q}^2$ such that $p\in(r,s)$ and $u\in (f(r,y;p),f(r,z;p)).$ Then 
	$f(p,u;t)\leq f(r,z;t)<c,$ $f(p,u;s)=x$ and the point $(p,u)$ is the needed one. The relation \eqref{eq_5_4} and the lemma are proved.

\end{proof}

Define the group of shifts 
$$
\theta_h: \mathbb{F}\to\mathbb{F}, \  \theta_h(f)(s,x;t)=f(s+h,x;t+h),
$$
and a mapping
$$
\varphi:\mathbb{R}_+\times\mathbb{F}\times\mathbb{R}\to\mathbb{R}, \  \varphi(t,f,x)=f(0,x;t).
$$
The condition {\bf F1} of the defintion \ref{def2} implies that $\varphi$ is a perfect cocycle over $\theta.$

\begin{corollary} Mappings 
	$$
	\theta:\mathbb{R}\times\mathbb{F}\to\mathbb{F}, \  \varphi:\mathbb{R}_+\times\mathbb{F}\times\mathbb{R}\to\mathbb{R}
	$$
	are (jointly) measurable in all arguments, i.e. conditions {\bf RDS1, RDS2} of the definition \ref{def12} are satisfied.
\end{corollary}
The probability measure $\mu$ on the space $(\mathbb{F},\mathcal{A})$ that transforms $(\mathbb{F},\mathcal{A},\mu;\{\theta_h,h\in\mathbb{R}\}$ into a metric dynamical system and $\varphi$ into a random dynamical system is constructed in the next section.

\section{Random dynamical system}
A probability measure $\mu$ on the space $(\mathbb{F},\mathcal{A})$ will be constructed from a sequence of transition probabilities $\{P^{(n)}:n\geq 1\}$ that satisfy conditions {\bf TP1-TP6} (see the Introduction and formulation of the theorem \ref{thm1}). Our strategy is to build a flow $\{\psi_{s,t}:s\leq t\}$ generated by $\{P^{(n)}:n\geq 1\}$ in such a way that $\psi\in \mathbb{F}$ for all realizations. In such a way we will construct a random element $\psi$ in  $(\mathbb{F},\mathcal{A}).$  The measure $\mu$ will be defined as a distribution of $\psi.$

To build a stochastic flow $\psi$ we use a common  approach \cite{FINR, TW} of defining a partial stochastic flow starting only from a dense set of points of $\mathbb{R}^2$ -- the so-called skeleton of the flow. When defining a skeleton we must simultaneously build processes that start at different moments of time. For convenience we will define processes for all moments of time by simply make them constants before their real moments of start. 

\begin{lemma}
	\label{lem2} On a suitable probability space $(\Omega,\mathcal{F},\mathbb{P})$ there is a family of continuous processes $\{Y_{(s,u)}:(s,u)\in\mathbb{R}^2\}$ with the following properties
	
	\begin{enumerate}
		\item For all $(s,u)\in \mathbb{R}^2$ and $t\leq s$
		$$
		\mathbb{P}(Y_{(s,u)}(t)=u)=1.
		$$
		
		\item Denote by $\mathcal{F}^Y_s$ the ``past'' $\sigma-$field at a moment $s,$ $		\mathcal{F}^Y_s=\sigma(\{Y_{(r,u)}(t):(r,u)\in\mathbb{R}^2, t\leq s\}).$ 		Then for all $(s_1,u_1),\ldots,(s_n,u_n)\in\mathbb{R}^2,$ $s\geq s_1\vee \ldots \vee s_n,$ $t\geq s$ and $B\in\mathcal{B}(\mathbb{R}^n)$ one has 
		
	\end{enumerate}
	$$
	\mathbb{P}  \Big((Y_{(s_1,u_1)}(t),\ldots,Y_{(s_n,u_n)}(t))\in B|\mathcal{F}^Y_s\Big) 
	=P^{(n)}_{t-s}\Big((Y_{(s_1,u_1)}(s),\ldots,Y_{(s_n,u_n)}(s)),B\Big).
	$$

\end{lemma}

\begin{proof} Existence of the needed family of processes will be derived from the Kolmogorov's theorem \cite[Ch. II, Th. (31.1)]{RW} applied to the space $\mathcal{C}$ of all continuous functions $f:\mathbb{R}\to \mathbb{R}$  with the metric of uniform convergence on compacts and Borel $\sigma-$field. Denote by $\{Y^{(n)}(t):t\in\mathbb{R}\}$ the canonical process on $\mathcal{C}^n$
	and by $\{\mathcal{B}^{(n)}_t:t\geq 0\}$ the natural filtration on $\mathcal{C}^{(n)},$
	$$
	\mathcal{B}^{(n)}_t=\sigma(\{Y^{(n)}(s):s\leq t\}).
	$$
	For the proof it is enough to build a consistent family of probability measures
	$$
	\{\nu_{(s_1,u_1),\ldots,(s_n,u_n)}:n\geq 1,\{(s_1,u_1),\ldots,(s_n,u_n)\}\subset \mathbb{R}^2 \}
	$$
	indexed by all finite subsets of $\mathbb{R}^2,$ such that each measure $\nu_{(s_1,u_1),\ldots,(s_n,u_n)}$  satisfies two properties.
	
	\begin{enumerate}
		\item For every $i\in \{1,\ldots,n\}$	and $t\leq s_i$
		$$
		\nu_{(s_1,u_1),\ldots,(s_n,u_n)}(Y^{(n)}_i(t)=u_i)=1.
		$$
		\item For all $1\leq i_1<\ldots<i_k\leq n,$ $s\geq s_{i_1}\vee \ldots\vee s_{i_k},$ $t\geq s$ and $B\in\mathcal{B}(\mathbb{R}^k)$ one has
	\end{enumerate}
	
	$$
	\mathbb{P}  \Big((Y^{(n)}_{i_1}(t),\ldots,Y^{(n)}_{i_k}(t))\in B|\mathcal{B}^{(n)}_s\Big) 
	=P^{(k)}_{t-s}\Big((Y^{(n)}_{i_1}(s),\ldots,Y^{(n)}_{i_k}(s)),B\Big).
	$$

	\begin{remark}
		\label{rem30_1} Computations given in the appendix  show that properties 1 and 2  uniquely determine the probability measure $\nu_{(s_1,u_1),\ldots,(s_n,u_n)}.$
	\end{remark}
	
	Measures $\nu_{(s_1,u_1),\ldots,(s_n,u_n)}$ will be built by induction. Recall that $\{\mathbb{P}^{(n)}_x,x\in\mathbb{R}^n\}$ are families of distributions that correspond to continuous Markov processes with transition probabilities $\{P^{(n)}_t,t\geq 0\}.$ For each pair $(s,u)\in \mathbb{R}^2$ denote by $\nu_{(s,u)}$ the image of the measure $\mathbb{P}^{(1)}_u$ on $C([0,\infty),\mathbb{R})$ under  the mapping
	$$
	T^{(1)}_{s}:C([0,\infty),\mathbb{R})\to \mathcal{C}, \ T^{(1)}_{s}(f)(t)=f(t\vee s-s),
	$$
	$$
	\nu_{(s,u)}=\mathbb{P}^{(1)}_u \circ (T^{(1)}_{s})^{-1}.
	$$
	The measure $\nu_{(s,u)}$ obviously satisfies properties 1 and 2.
	Assume that a probability measure $\nu_{(s_1,u_1),\ldots,(s_{n-1},u_{n-1})}$ is constructed for each $(n-1)-$tuple of points $(s_1,u_1),$ $\ldots,$ $(s_{n-1},u_{n-1})$ and satisfies properties 1 and 2. Consider an $n-$tuple $(s_1,u_1),\ldots,(s_{n},u_{n})\in \mathbb{R}^2$ and let $m$ be the index of maximal  $s_i$
	$$
	m\in \arg\max_{1\leq i\leq n}s_i.
	$$
	Then define 
	$$
	(r_i,v_i)=\begin{cases}
	(s_i,u_i), \ 1\leq i\leq m-1 \\
	(s_{i+1},u_{i+1}), \ m\leq i\leq n-1
	\end{cases}
	$$
	so that the measure $\nu_{(r_1,v_1),\ldots,(r_{n-1},v_{n-1})}$ is already defined. Consider mappings
	$$
	T^{(n)}_r:\mathcal{C}^{n-1}\times C([0,\infty),\mathbb{R}^n)\to \mathcal{C}^n,
	$$
	defined for $t\leq r$ by
	$$
	T^{(n)}_r(f_1,\ldots,f_{n-1},g_1,\ldots,g_n)(t)=
	$$
	$$
	=(f_1(t),\ldots,f_{m-1}(t),g_m(0),f_{m}(t),\ldots,f_{n-1}(t))
	$$
	and for $t> r$ by
	$$
	T^{(n)}_r(f_1,\ldots,f_{n-1},g_1,\ldots,g_n)(t)=
	$$
	$$
	=(g_1(t-r)-g_1(0)+f_1(0),\ldots,g_{m-1}(t-r)-g_{m-1}(0)+f_{m-1}(0),g_m(t-r),
	$$
	$$
	g_{m+1}(t-r)-g_{m+1}(0)+f_{m}(0),\ldots,g_{n}(t-r)-g_{n}(0)+f_{n-1}(0)).
	$$
	Further, on the space $\mathcal{C}^{n-1}\times C([0,\infty),\mathbb{R}^n)$ we take the mixture of measures
	$$
	\alpha(df,dg)=\nu_{(r_1,v_1),\ldots,(r_{n-1},v_{n-1})}(df)\mathbb{P}^{(n)}_{(f_1(s_m),\ldots,f_{m-1}(s_m),u_m,f_{m}(s_m),\ldots,f_{n-1}(s_m))}(dg)
	$$
	and define $\nu_{(s_1,u_1),\ldots,(s_{n},u_{n})}$ as the image of the measure $\alpha$ under the mapping $T^{(n)}_{s_m},$
	$$
	\nu_{(s_1,u_1),\ldots,(s_{n},u_{n})}=\alpha\circ (T^{(n)}_{s_m})^{-1}.
	$$
	Properties 1 and 2 for the measure $\nu_{(s_1,u_1),\ldots,(s_n,u_n)}$ follow from the Markov property of measures $\mathbb{P}^{(n)}$ and consistency condition {\bf TP2}.
	
\end{proof}

Built processes $\{Y_{(s,u)}, (s,u)\in\mathbb{R}^2$ are finite-point motions from the flow that we aim to build. Let us enumerate points of $\mathbb{Q}^2$ in a sequence $\{(p_i,u_i):i\geq 1\}.$  Processes $\{Y_{(p_i,u_i)}:$ $i\geq 1\}$ will be used as as a skeleton of the stochastic flow $\psi$. In the next lemma we choose a suitable version of the skeleton that will give rise to an $\mathbb{F}-$valued modification of $\psi.$

\begin{lemma}
	\label{lem3} Let $\{Y_{(p_i,u_i)}:$ $i\geq 1\}$ be a skeleton flow constructed in the lemma \ref{lem2}. 	Following properties hold out of the set of measure zero. 
	
	{\bf SP1} For all $i,j$ such that $p_j>p_i,$ $	Y_{(p_i,u_i)}(p_j)\ne u_j.$
	
	{\bf SP2} Let $\tau_{ij}=\inf\{t\geq s_i\vee s_j: Y_{(s_i,u_i)}(t)=Y_{(s_j,u_j)}(t)\}$ be the meeting time of processes $Y_{(s_i,u_i)}$ and $Y_{(s_j,u_j)}.$ Then  for all $i\ne j$
	$$
	Y_{(s_i,u_i)}(t)=Y_{(s_j,u_j)}(t), \ t\geq \tau_{ij}.
	$$
	
	{\bf SP3} For all $s\in \mathbb{R}$ the range $\mathcal{R}_s(Y)=\{Y_{(p_i,u_i)}(s):p_i<s\}$ is dense in $\mathbb{R}.$
	
	{\bf SP4} For all $s<t$ and $a<b$ the set $	\{Y_{(p_i,u_i)}(t):p_i\leq s, Y_{(p_i,u_i)}(s)\in(a,b)\} $	is finite.
	
	{\bf SP5} For all $i\geq 1$
	$$
	\lim_{u_j\to u_i+}\max_{t\geq p_i} (Y_{(p_i,u_j)}(t)-Y_{(p_i,u_i)}(t))= 0
	$$
	
\end{lemma}

\begin{proof} The property {\bf SP1} follows from {\bf TP6}. The coalescing condition {\bf TP3} and strong Markov property of two-point motions $\mathbb{P}^{(2)}$ \cite[Ch. 4, Th. 2.7]{EK} give the property {\bf SP2}.
	
	To prove the property {\bf SP3} we  will refine the condition {\bf TP4}. For each $u\in \mathbb{R}$ and $\varepsilon>0$ one has
	\begin{equation}
	\label{eq04_1}
	t^{-1}\mathbb{P}^{(1)}_u (\max_{r\in[0,t]}|X^{(1)}(r)-u|>\varepsilon)\to 0, \ t\to 0.
	\end{equation}
	Indeed, introduce the stopping time $\sigma(f)=\inf\{r\geq 0: |f(r)-f(0)|\geq \varepsilon\}.$ Given $\delta>0$ there exists $t_0=t_0(\delta)>0$ such that for all $t\leq t_0$
	$$
	P^{(1)}_t(u,(u-\varepsilon/2,u+\varepsilon/2)^c)\leq \delta t,
	$$
	$$
	P^{(1)}_t(u+\varepsilon,(u+\varepsilon/2,u+3\varepsilon/2)^c)\leq \delta t,
	$$
	$$
	P^{(1)}_t(u-\varepsilon,(u-3\varepsilon/2,u-\varepsilon/2)^c)\leq \delta t.
	$$
	Estimate the left-hand side of \eqref{eq04_1} as follows
	$$
	\begin{aligned}
	\mathbb{P}^{(1)}_u ( & \max_{r\in[0,t]}|X^{(1)}(r)-u|>\varepsilon)  \leq   P^{(1)}_t(u,(u-\varepsilon/2,u+\varepsilon/2)^c) \\
	& 	+\mathbb{P}^{(1)}_u (|X^{(1)}(t)-u|\leq \varepsilon/2, \max_{r\in[0,t]}|X^{(1)}(r)-u|>\varepsilon)\\ 
	\leq & P^{(1)}_t(u,(u-\varepsilon/2,u+\varepsilon/2)^c) \\
	& +
	\mathbb{E}_{\mathbb{P}^{(1)}_u} P^{(1)}_{t-\sigma}(X^{(1)}(\sigma),(X^{(1)}(\sigma)-\varepsilon/2,X^{(1)}(\sigma)+\varepsilon/2)^c)\leq 2\delta t, 
	\end{aligned}
	$$
	where on the last step we used the relation $X^{(1)}(\sigma)=X^{(1)}(0)\pm\varepsilon.$ Convergence \eqref{eq04_1} is proved.
	
	Let us consider a rectangle $[s_1,s_2]\times [a_1,a_2]$ where all points $a_1,a_2,s_1,s_2$ are rational. Denote $u=\frac{a_1+a_2}{2},$ $\delta=\frac{a_2-a_1}{2}.$ For each $n\geq 1$ choose $t_0>0$ such that 
	$$
	\mathbb{P}^{(1)}_u(\max_{r\in [0,t]}|X^{(1)}(r)-u|>\delta)\leq \frac{t}{2^n}, \ t\leq t_0.
	$$
	Consider a uniform partition  $	s_1=\xi_0<\xi_1<\ldots<\xi_{k_n}=s_2,$ 	where number of segments $k_n>\frac{2(s_2-s_1)}{t_0}.$ Let  $A_n$ be the event
	$$
	A_n=\bigcup^{k_n-2}_{i=0}\{\max_{r\in [\xi_i,\xi_{i+2}]}|Y_{(\xi_i,u)}(r)-u|>\delta\}.
	$$
	The probability of $A_n$ can be estimated as
	$$
	\mathbb{P}(A_n)\leq (k_n-1)\mathbb{P}^{(1)}_u(\max_{0\leq r\leq \frac{2(s_2-s_1)}{k_n}}|X^{(1)}(r)-u|>\delta)\leq \frac{2(s_2-s_1)(k_n-1)}{k_n 2^n}.
	$$
	Hence, $	\sum^\infty_{n=2} \mathbb{P}(A_n)<\infty$ 	and with probability 1 each segment $\{r\}\times (a_1,a_2),$ $s_1<r<s_2,$ intersects the range $\mathcal{R}_r(Y)$ (fig. 5).
	
	\begin{figure}[!h]
		\centering
		\caption{Interval $r_1\times(a_1,a_2)$ is intersected by trajectories of the skeleton while interval $r_2\times(a_1,a_2)$ is not.}
		\includegraphics[width=8cm]{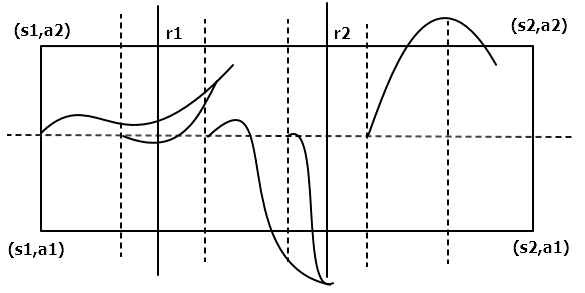}
	\end{figure} 
	
	Taking intersection in all rectangles $[s_1,s_2]\times[a_1,a_2]$ with rational vertices gives the needed result.
	
	The proof of {\bf SP4} is in lines of \cite{D2, Matsumoto}. Consider $n\geq 2,$ $a<b,$ points $a<u_1<\ldots<u_n<b$ and a continuous function $f \in C([0,\infty),\mathbb{R}^n)$ starting at $u=(u_1,\ldots,u_n).$ For  fixed time $t>0$ denote by $d_n(t,f)$ the number of points in the image $\{f_1(t),\ldots,f_n(t)\}:$ 
	$$
	d_n(t,f)=|\{f_1(t),\ldots,f_n(t)\}|.
	$$
	Expectation of $d_n(t,X^{(n)})$ with respect to $\mathbb{P}^{(n)}_{u}$ is estimated using {\bf TP5}:
	$$
	\begin{aligned}
	\mathbb{E}_{\mathbb{P}^{(n)}_{u}}& d_n(t,X^{(n)})1_{\{\forall s\in[0,t] \ (X^{(n)}_1(s),X^{(n)}_n(s))\in [c,c']^2\}} \\
	\leq & \mathbb{E}_{\mathbb{P}^{(n-1)}_{(u_1,\ldots,u_{n-1})}}d_{n-1}(t,X^{(n-1)})1_{\{\forall s\in[0,t] \ (X^{(n-1)}_1(s),X^{(n-1)}_{n-1}(s))\in [c,c']^2\}} \\
	&	+\mathbb{P}^{(2)}_{(u_{n-1},u_n)}(\forall s\in[0,t] \ (X^{(2)}_1(s),X^{(2)}_2(s))\in[c,c']^2 \mbox{ and } X^{(2)}_1(s)\ne X^{(2)}_2(s))\\
	\leq & \mathbb{E}_{\mathbb{P}^{(n-1)}_{(u_1,\ldots,u_{n-1})}}d_{n-1}(t,X^{(n-1)})1_{\{\forall s\in[0,t] \ (X^{(n-1)}_1(s),X^{(n-1)}_{n-1}(s))\in [c,c']^2\}}\\
	+& m(u_n)-m(u_{n-1})	\leq  \ldots \\
	\leq & 1+ \sum^{n-1}_{i=1}(m(u_{i+1})-m(u_i))\leq 1+m(b)-m(a),
	\end{aligned}
	$$
	where $m$ is an increasing function from the condition {\bf TP5}. $m$ does not depend on the choice of points $u_1,\ldots,u_n.$ 
	
	Now let us consider rational points $s<t$ and $a<b.$ For each $n\geq 1$ we define a random set of positions of first $n$ trajectories from the skeleton that passed through  $(a,b)$ at time $s,$
	$$
	C_n=\{Y_{(p_i,u_i)}(s):1\leq i\leq n, p_i\leq s, a<Y_{(p_i,u_i)}(s)<b\}.
	$$
	Let $k_n$ be the cardinality of $C_n,$ i.e. $	C_n=\{\pi_1,\ldots,\pi_{k_n}\}, $ 	where $\pi_1<\ldots<\pi_{k_n}.$ Denote by $\xi_1,\ldots,\xi_{k_n}$ parts of trajectories of the skeleton that passed through points $\pi_1,\ldots,\pi_{k_n},$ 
	$$
	\xi_{j}(t)=Y_{(p_i,u_i)}(s+t), \ t\geq 0,
	$$
	where $i$ is chosen in such a way that $Y_{(p_i,u_i)}(s)=\pi_j$ (definition is correct in the view of {\bf SP2}). Finally, let $\zeta_n$ be the number of points in the image $\{\xi_1(t-s),\ldots,\xi_{k_n}(t-s)\}:$ $	\zeta_n=d_{k_n}(t-s,\xi).$
	
	From above computations it follows that 
	$$
	\begin{aligned}
	\mathbb{E}&  \zeta_n 1_{\{\forall r\in[s,t] \ (Y_{(s,a)}(r),Y_{(s,b)}(r)\in [c,c']^2\}} \\
	\leq & \mathbb{E} d_{k_n}(t-s,\xi)1_{\{\forall r\in[0,t-s] \ (\xi_1 (r), \xi_{k_n}(r))\in[c,c']^2\}}\leq 1+m(b)-m(a).
	\end{aligned}
	$$
	Then for all $c<c'$
	$$
	\mathbb{E} (\sup_{n\geq 1}\zeta_n) 1_{\{\forall r\in[s,t] \ (Y_{(s,a)}(r),Y_{(s,b)}(r)\in [c,c']^2\}}\leq 1+m(b)-m(a).
	$$
	From the continuity of $Y_{(s,u)}$ events $\{\forall r\in[s,t] \ (Y_{(s,a)}(r),Y_{(s,b)}(r)\in [c,c']^2\}$ increase to $\Omega$ as $c\to-\infty,$ $c'\to \infty.$ So, with probability 1 $	\sup_{n\geq 1}\zeta_n <\infty.$ 
	To prove {\bf SP4} it remains to take union in all quadruples of rational points $s<t,$ $a<b.$
	
	Consider a point $(p_i,u_i)\in \mathbb{Q}^2,$ some $q>p_i,$ rational $v>u_i$ and $c<c'.$ Let $m$ be a continuous increasing function from the condition {\bf TP5} that corresponds to $c,c'$ and $q-p_i.$ Find a sequence ${j_n},$ such that $p_{j_n}=p_i,$  $v>u_{j_n}\downarrow u_i$ and
	$$
	\sum^\infty_{n=1} (m(u_{j_n})-m(u_i))<\infty.
	$$
	By the Borel-Cantelli lemma with probability $1$ a condition 
	$$
	\forall s\in[p_i,q] \  (Y_{(p_i,u_i)}(s),Y_{(p_i,v)}(s)\in [c,c']^2
	$$
	implies the existence of an index $n$  such that $Y_{(p_i,u_{j_n})}(q)=Y_{(p_i,u_i)}(q).$ 	Taking union in all $c,c'$ we obtain that with probability $1$ for any point $(p_i,u_i)\in \mathbb{Q}^2$ and any $q>p_i$ there exists rational $u_j>u_i,$ such that 
	$$
	Y_{(p_i,u_{j})}(q)=Y_{(p_i,u_i)}(q).
	$$
	Now  given $\varepsilon>0$ consider $k,$ such that  $u_i<u_{k}<u_i+\varepsilon.$ There exists $q>p_i,$ such that 
	$$
	u_i-\varepsilon<Y_{(p_i,u_i)}(r)\leq Y_{(p_i,u_k)}(r)<u_i+\varepsilon \mbox{ for all } r\in [p_i, q].
	$$
	As it was shown above, for some $u_{j_n}\in (u_i,u_k)$ we have $	Y_{(p_i,u_{j_n})}(q)=Y_{(p_i,u_i)}(q).$ Then for all $u_j\leq u_{j_n}$ and all $r\geq p_i$
	$$
	Y_{(p_i,u_i)}(r)\leq Y_{(p_i,u_i)}(r)\leq Y_{(p_i,u_i)}(r)+2\varepsilon.
	$$
	{\bf SP5} is checked and the lemma \ref{lem3} is proved.
	
\end{proof}

Now we can define a version of a stochastic flow generated by transition probabilities $\{P^{(n)},n\geq 1\}$ which is a random element in $\mathbb{F}:$
$$
\psi(s,x;t)=\inf\{Y_{(p_i,u_i)}(t):p_i\leq s, Y_{(p_i,u_i)}(s)\geq x\}.
$$

Actually there are two options in the definition depending on whether $(s,x)$ lies on the skeleton or not (see fig. 6).

\begin{figure}[!h]
	\centering
	\caption{A point $(r,x)$ lies on the skeleton while a point $(r,y)$ does not.}
	\includegraphics[width=8cm]{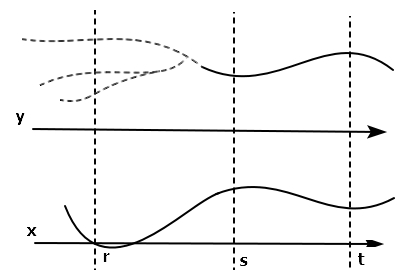}
\end{figure}

In the first case we continue the trajectory of the skeleton. In the second case we take a lower envelope of trajectories of the skeleton that passed above $x$ at time $s$. By properties {\bf SP4, SP5} of the lemma \ref{lem3} the definition is correct and gives continuous functions $t\to \psi(s,x;t)$.

The family $\psi$ satisfies all conditions of the definition \ref{def2}. Indeed, condition {\bf F1} follows from the characterization above: if $r<s<t$ then $\psi(r,x;s)=Y_{(p_i,u_i)}(s)$ 
for some $i$ with $p_i\leq r$ and $Y_{(p_i,u_i)}(r)=x.$ So,
$$
\psi(s,\psi(r,x;s);t)=Y_{(p_i,u_i)}(t)\geq \psi(r,x;t).
$$
But if $X_{(p_i,u_i)}(t)> \psi(r,x;t)$ then there exists $j$ with $p_j\leq r,$ $Y_{(p_j,u_j)}(r)\geq x$ and $Y_{(p_j,u_j)}(t)<Y_{(p_i,u_i)}(t).$ From {\bf SP2} it follows that $Y_{(p_j,u_j)}(s)<Y_{(p_i,u_i)}(s)=\psi(r,x;s)$ which is impossible. Property {\bf F2} immediately follows from {\bf SP3}.  Property {\bf F3} is the right-continuity of the mapping $x\to \psi(s,x;t)$ at a point $x\not\in \mathcal{R}_s(\psi).$ If a point $(s,x)$ does not lie on the graph of the skeleton, then the property {\bf SP4} and the definition of $\psi$ imply that $x\to \psi(s,x;t)$ is constant to the right of $x$. If a point $(s,x)$ lies on the graph of the skeleton and $x\not\in\mathcal{R}_s(\psi)$ then the only possibility is $(s,x)=(p_i,u_i)$ and right-continuity follows from {\bf SP5}. Finally, the condition {\bf F4} follows from {\bf SP1} and the definition of $\psi$: each trajectory $\psi(s,x;t)$ at time $t>s$ coincides with some $Y_{(p_i,u_i)}(t),$ and the point $u_i$ is a fresh point at time $p_i.$

Consequently, $\psi$ is a measurable mapping $\psi:\Omega\to \mathbb{F}$ and a measure $\mu=\mathbb{P}\circ \psi^{-1}$ is well-defined. The construction of a random dynamical system will be finished when we will check that $\mu\circ \theta^{-1}_h=\mu.$ This is done is the next lemma.

\begin{lemma}\label{lem4} For all $h\in \mathbb{R}$ $\mu\circ \theta^{-1}_h=\mu.$

\end{lemma}

\begin{proof} The measure $\mu$ is the distribution in $\mathbb{F}$ of the stochastic flow 
	$$
	\psi(s,x;t)=\inf\{Y_{(p_i,u_i)}(t):p_i\leq s, Y_{(p_i,u_i)}(s)\geq x\}.
	$$
	Then $\mu\circ \theta^{-1}_h$ is the distribution of the stochastic flow $\psi_h(s,x;t)=\psi(s+h,x;t+h).$ 	It is enough to check coincidence of finite dimensional distributions of these two flows. In turn it will follow from condition {\bf SF3} of the definition \ref{def29_1}. Observe that the past $\sigma-$field 
	$$
	\mathcal{F}^\psi_s=\sigma(\{\psi(p,x;r):p\leq r\leq s, x\in\mathbb{R}\})
	$$
	coincides with 
	$$
	\mathcal{F}^Y_s=\sigma(\{Y_{(p_i,u_i)}(r):i\geq 1, p_i\leq r\leq s\}).
	$$
	We will check that for any $\mathcal{F}^Y_s$-measurable random variables $\xi_1,\ldots,\xi_n$ and any $c_1,\ldots,c_n\in \mathbb{R}$
	$$
	\mathbb{P}\Big(\psi(s,\xi_1;t)<c_1,\ldots,\psi(s,\xi_n;t)<c_n|\mathcal{F}^Y_s\Big)=P^{(n)}_{t-s}\bigg((\xi_1,\ldots,\xi_n);\prod^n_{i=1}(-\infty,c_i)\bigg).
	$$
	To do it we will organize an $\mathcal{F}^Y_s$-measurable approximation procedure of starting points $\xi_1,\ldots,\xi_n.$
	
	For each partition $\Pi$ of the set $\{1,\ldots,n\}$ introduce an event $$
	A_{\Pi}=\bigcap^n_{i,j=1}\{\xi_i=\xi_j\Leftrightarrow i=j \mbox{ mod }\Pi\}.
	$$
	Following constructions are perfomed on the set $A_{\Pi}.$ Denote by $(\eta_1,\ldots,\eta_k)$ values of $(\xi_1,\ldots,\xi_n).$  Here $k=|\Pi|$ and $\eta_1<\ldots<\eta_k.$ Denote elements of the set $\{i\in\{1,\ldots,n\}:\xi_i=\eta_j\}$ by $\{i^j_1,\ldots,i^j_{d_j}\}.$  For each $m\geq 1$ consider 
	$$
	y^m_{j}=\min\{Y_{(p_i,u_i)}(s):1\leq i\leq m, p_i\leq s, Y_{(p_i,u_i)}(s)\geq \eta_j\},
	$$
	where the minimum of the empty set is defined as $\infty.$ Then for each $j\in \{1,\ldots,k\}$ one has $y^m_j\to \eta_j+$ and $\psi(s,y^m_j;t)\to \psi(s,\eta_j;t)+,$ $m\to\infty$ (here we also agree that $\psi(s,\infty;t)=\infty$). Indeed, there are two possibilities (see fig. 6). If a point $(s,\eta_j)$ lies on the skeleton, then at some instant we will have coincidence $y^m_j=\eta_j$ and, respectively, $\psi(s,y^m_j;t)=\psi(s,\eta_j;t).$ Otherwise, right-continuity gives the needed convergence.
	
	So, on the set $A_\Pi$ one has
	$$
	\begin{aligned}
	\mathbb{P}& (\psi(s,\xi_1;t)<c_1,\ldots,\psi(x,\xi_n;t)<c_n|\mathcal{F}^Y_s) \\
	=& \lim_{m\to\infty}\mathbb{P}(\psi(s,y^m_1;t)<c_{i^1_1}\wedge\ldots\wedge c_{i^1_{d_1}},\ldots,\psi(x,y^m_k;t)<c_{i^k_1}\wedge\ldots\wedge c_{i^k_{d_k}}|\mathcal{F}^X_s).
	\end{aligned}
	$$
	The latter probability is found from the conditional distribution of the vector $(Y_{(p_1,u_1)}(t),$ $\ldots,Y_{(p_m,u_m)}(t))$ with respect to $\mathcal{F}^Y_s.$ By lemma \ref{lem2} and consistency condition {\bf TP1} one can write the limit as 
	$$
	\begin{aligned}
	\lim_{m\to\infty} & P^{(k)}_{t-s}\bigg((y^m_1,\ldots,y^m_k);\prod^k_{j=1}(-\infty,c_{i^j_1}\wedge\ldots\wedge c_{i^j_{d_j}})\bigg) \\
	= & P^{(k)}_{t-s}\bigg((\eta_1,\ldots,\eta_k);\prod^k_{j=1}(-\infty,c_{i^j_1}\wedge\ldots\wedge c_{i^j_{d_j}})\bigg) \\
	= & P^{(k)}_{t-s}\bigg((\xi_1,\ldots,\xi_n);\prod^n_{i=1}(-\infty,c_i)\bigg).
	\end{aligned}
	$$
	On the last step we used continuity of a distribution of a Feller process in space variable, absence of atoms, consistency condition {\bf TP2} and coalescing condition {\bf TP3}. The lemma is proved.
\end{proof}

Results of this section show that the stochastic flow of mappings defined by transition probabilities satisfying properties {\bf TP1-TP6} gives rise to a random dynamical  system $\varphi$ over the space of flows $(\mathbb{F},\mathcal{A},\mu,\theta)$  in the sense of definition \ref{def12}.

\section{Examples}

\subsection{Coalescing diffusion processes}

In this section we construct a random dynamical system defined by a flow of coalescing diffusion processes, that are independent before the meeting time. At first we describe corresponding transition probabilities $\{P^{(n)}:n\geq 1\}.$ Consider a stochastic differential equation
\begin{equation}
\label{eq04_5}
dX(t)=a(X(t))dt+b(X(t))dw(t)
\end{equation}
where $w$ is a Wiener process and coefficients $a,b:\mathbb{R}\to\mathbb{R}$ are globally Lipschitz,  $b(x)>0$  for all $x\in\mathbb{R}.$

It is well-known \cite[Ch. V, Th. (11.2)]{RW2} that for every $x\in\mathbb{R}$ the equation \eqref{eq04_5} has a unique strong solution $\{X_x(t):t\geq 0\},$ $X_x(0)=x.$ Moreover, measures 
\begin{equation}
\label{eq04_6}
P^{(1)}_t(x,A)=\mathbb{P}(X_x(t)\in A)
\end{equation}
constitute a Feller semigroup of transition probabilities on $\mathbb{R}$ \cite[Ch. V, Th. (24.1)]{RW2}. Condition {\bf TP4} follows from the estimate $P^{(1)}_t(x,(x-\varepsilon,x+\varepsilon))\geq T^{(1)}_t h_\varepsilon(x),$ where $h_\varepsilon:\mathbb{R}\to[0,1]$ is an infinitely differentiable function that equals $1$ in a neighborhood of $x$ and has a support in $(x-\varepsilon,x+\varepsilon).$ Condition {\bf TP6} holds because measures $P^{(1)}_t(x,\cdot),$ $t>0,$ are absolutely continuous with respect to the Lebesgue measure \cite[Cor. 10.1.4]{SV}.

Denote by $P^{(n),ind.}$ a transition probability that corresponds to $n$ independent solutions of \eqref{eq04_5},
\begin{equation}
\label{eq04_7}
P^{(n),ind.}_t((x_1,\ldots,x_n),A_1\times \ldots \times A_n)=\prod^n_{i=1} P^{(1)}_t(x_i,A_i).
\end{equation}
Obviously, transition semigroups $\{P^{(n),ind.}:n\geq 1\}$ are Feller and consistent, i.e. conditions {\bf TP1, TP2} are satisfied. 

In the next lemma we show that condition {\bf TP5} holds for $P^{(2),ind.}.$ Further, from a general result of \cite{LJR} we obtain a sequence of (coalescing) transition semigroups $\{P^{(n)}:n\geq 1\}$ that satisfy conditions {\bf TP1-TP6} and are such that distributions of the stopped process  $X^{(n)}(\tau_n\wedge \cdot)$ (where $\tau_n=\inf\{t\geq 0: X^{(n)}_i(t)=X^{(n)}_j(t) \mbox{ for some } i\ne j\})$ under measures $\mathbb{P}^{(n)}_{(x_1,\ldots,x_n)}$ and $\mathbb{P}^{(n),ind.}_{(x_1,\ldots,x_n)}$ coincide. Transition semigroups$\{P^{(n)}:n\geq 1\}$ desribe the needed dynamics: any $n$-tuple of processes move like independent solutions to \eqref{eq04_5} and coalesce at the moment of meeting.

\begin{lemma}
	\label{lem5} For each $c<c'$ and $t>0$ there exists a continuous increasing function $m:\mathbb{R}\to\mathbb{R}$ such that for all $x,y$
	$$
	\begin{aligned}
	\mathbb{P}^{(2),ind.}_{x,y}(& \forall s\in[0,t] \ (X^{(2)}_1(s),X^{(2)}_2(s))\in[c,c']^2\\
	& \mbox{ and } X^{(2)}_1(s)\ne X^{(2)}_2(s))\leq |m(x)-m(y)|.
	\end{aligned}
	$$
\end{lemma}

\begin{proof} Assume that $c\leq x\leq y\leq c'.$ At first we consider the case $a=0.$ Denote by $\tilde{b}$ a bounded Lipschitz function such that $\tilde{b}(x)\geq \delta$ for some $\delta>0$ and all $x\in\mathbb{R},$ and $b(x)=\tilde{b}(x)$ for $x\in[c,c'].$ Then we can estimate
	$$
	\begin{aligned}
	\mathbb{P}^{(2),ind.}_{x,y}(& \forall s\in[0,t] \ (X^{(2)}_1(s),X^{(2)}_2(s))\in[c,c']^2\\
	& \mbox{ and } X^{(2)}_1(s)\ne X^{(2)}_2(s)) \leq \mathbb{P}(\forall s\in[0,t] \  X(s)\ne Y(s)),	
	\end{aligned}
	$$
	where 
	$$
	\begin{cases}
	dX(t)=\tilde{b}(X(t))dw_1(t) \\
	X(0)=x
	\end{cases} \  	\begin{cases}
	dY(t)=\tilde{b}(Y(t))dw_2(t) \\
	Y(0)=y
	\end{cases} 
	$$
	and $w_1$ and $w_2$ are independent Wiener processes \cite[Ch. V, Cor. (11.10)]{RW2}. The difference $D(t)=Y(t)-X(t)$ is a continuous martingale with quadratic variation 
	$$
	\langle D\rangle(t)=\int^t_0 (\tilde{b}(X(s))^2+\tilde{b}(Y(s))^2)ds\geq 2\delta^2 t
	$$
	By \cite[Ch. V, Th. (1.6)]{RY} $D$ can be written as $D(t)=y-x+B(\langle D\rangle(t))$ with $B$ being a Brownian motion. So,
	$$
	\begin{aligned}
	\mathbb{P} &( \forall s\in[0,t] \  X(s)\ne Y(s))=\mathbb{P}(\forall s\in[0,t]\  D(t)>0) \\
	= & \mathbb{P}(\forall s\in[0,\langle D\rangle(t)]\  B(s)>x-y) \\
	\leq & \mathbb{P}(\forall s\in[0,2\delta^2 t]\  B(s)>x-y)=\sqrt{\frac{2}{\pi}}\int^{\frac{y-x}{\sqrt{2t}\delta}}_0e^{-z^2/2}dz\leq \frac{y-x}{\sqrt{\pi t}\delta}.
	\end{aligned}
	$$
	For a deep and unified study of similar inequalities for more general stochastic flows we refer to \cite{D_ent}.
	
	The general case follows via the transformation
	$$
	m(x)=\int^x_0\exp\bigg(-2\int^y_0\frac{a(z)}{b^2(z)}dz\bigg)dy
	$$
	that removes the drift \cite{Zvonkin}.

\end{proof}

\begin{corollary}\label{cor29_01} There exists a unique sequence $\{P^{(n)}:n\geq 1\}$ of transition probabilities that satisfy conditions {\bf TP1-TP6} and are such that distributions of the stopped process   $X^{(n)}(\tau_n\wedge \cdot)$ (where $\tau_n=\inf\{t\geq 0: X^{(n)}_i(t)=X^{(n)}_j(t) \mbox{ for some } i\ne j\})$ under measures $\mathbb{P}^{(n)}_{(x_1,\ldots,x_n)}$ and $\mathbb{P}^{(n),ind.}_{(x_1,\ldots,x_n)}$ (given by \eqref{eq04_5},\eqref{eq04_6},\eqref{eq04_7}) coincide. 
\end{corollary}
\begin{proof} By  \cite[Th. 4.1]{LJR} it is enough to check that for all $t>0$ and $x\in\mathbb{R}$
	$$
	\mathbb{P}^{(2),ind.}_{x,y}(\forall s\in[0,t] \ X^{(2)}_1(s)\ne X^{(2)}_2(s))\to 0, \ y\to x.
	$$
	But the latter convergence follows from the property {\bf TP5} and the continuity of trajectories.

\end{proof}

As a conlcusion, one can build a random dynamical system $\varphi:\mathbb{R}_+\times\Omega\times \mathbb{R}\to \mathbb{R}$ over certain measurable dynamical system $(\Omega,\mathcal{F},\mathbb{P},\theta)$ such that for any starting points $x$ the process $t\to \varphi(t,x)$ is distributed as a solution to \eqref{eq04_5} and for any starting points $x_1<\ldots<x_n$ the process $t\to (\varphi(t,x_1),\ldots,\varphi(t,x_n))$ is a family of coalescing processes independent before the first meeting time. As partial cases we get following examples.

\begin{example}\label{ex29_1}
	If $a=0, b=1,$ then a random dynamical system corresponds to the Arratia flow.
\end{example}

\begin{example}\label{ex29_2}
	If $a(x)=-\lambda x, b=\sigma,$ then a random dynamical system corresponds to the flow of coalescing Ornstein-Uhlenbeck processes. Ergodic properties of such systems will be studied in our future work.
\end{example}

From a naive point of view described models can be viewed as flows of solutions to ``stochastic differential equation''
$$
\begin{cases}
dX(u,t)=a(X(u,t))dt+b(X(u,t))dw_{t,X(u,t)}(t) \\
X(u,0)=u
\end{cases}
$$ ``driven'' by the Arratia flow $\{w_{(s,x)}(t):s\leq t, x\in\mathbb{R}\}$. However, a problem of a rigorous definition of such equations remains open. For some related questions we refer to \cite{DVovch}.

\subsection{Coalescing Harris flows}

More examples of transition semigroups $\{P^{(n)}:n\geq 1\}$ that satisfy {\bf TP1-TP6} and give rise to random dynamical systems are provided by Harris flows.  Consider a positive definite function 
$$
\Gamma:\mathbb{R}\to\mathbb{R},
$$
that satisfy following conditions

\begin{enumerate}
	\item $\Gamma$ is a characteristic function of some probability measure $\mu$ of infinite support;

	\item there exists continuous function $\beta:(0,1]\to (0,\infty)$ such that 
	$$
	1-\Gamma(x)\geq \beta(x);
	$$
	
	\item for some $\varepsilon>0$
	$$
	\int^\varepsilon_0 \frac{xdx}{\beta(x)}<\infty;
	$$
	
	\item for some $\varepsilon>0,$ $\frac{\beta(x)}{x^2}$ is monotone decreasing on $(0,\varepsilon).$
	
\end{enumerate}

In \cite{Harris} it is proved that there exists a family $\{w_{x}(t):t\geq 0, x\in\mathbb{R}\}$ of Wiener process (with respect to the joint filtration), such that $w_x(0)=x$ and 
$$
d\langle w_x,w_y\rangle(t)=\Gamma(w_x(t)-w_y(t))dt.
$$
Such flow is a ``correlated'' analogue of the Arratia flow. Condition 3 implies that the coalescence is present in the flow. Related transition semigroups are defined by 
$$
P^{(n)}_t((x_1,\ldots,x_n),B)=\mathbb{P}((w_{x_1}(t),\ldots,w_{x_n}(t))\in B).
$$ 
They satisfy conditions {\bf TP1-TP6} as follows from \cite{Harris} and \cite[Th. 3.2]{Matsumoto}. Hence, random dynamical system can be built for such Harris flows.

\section{Appendix. On definitions of a stochastic flow}

Here we discuss differences in our definition \ref{def29_1} of the stochastic flow and more usual definition \cite[Def. 1.6]{LJR}. The difference is caused by the fact that we consider stochastic flows of discontinuous random mappings. Consequently, composition of such mappings may not possess good measurability properties. 

Consider a stochastic flow $\{\psi_{s,t}(x):-\infty<s\leq t<\infty,x\in \mathbb{R}\}$ in the sense of the definition \ref{def29_1}. In terms of the semigroup $T^{(n)}$ our condition {\bf SF3} states that for any $s\leq t,$ any $\mathcal{F}^\psi_{s}-$measurable $\mathbb{R}^n-$valued random  vector $\xi$ and any function $f\in C_0(\mathbb{R}^n)$ one has
\begin{equation}
\label{eq22_1}
\mathbb{E}(f(\psi_{s,t}(\xi))/\mathcal{F}^\psi_s)=T^{(n)}_{t-s}f(\xi).
\end{equation}
Together with {\bf SF1, SF2} it allows to recover all finite-dimensional distributions of the family $\{\psi_{s,t}(x):-\infty<s\leq t<\infty, x\in \mathbb{R}\}$ from transition probabilities $\{P^{(n)}:n\geq 1\}.$ Indeed, consider  points $t_1\leq t_2\leq \ldots\leq t_n,$ $x_1,\ldots,x_l\in \mathbb{R}$ and $l\frac{n(n-1)}{2}$ functions  $g^{(k)}_{i,j}\in C_0(\mathbb{R}),$ $1\leq k\leq l, 1\leq i<j\leq n.$ Next identities follow from conditions {\bf SF1, SF2, SF3} and \eqref{eq22_1} applied to $\mathcal{F}^\psi_{t_{n-1}}-$measurable random vector $\xi=(\psi_{t_i,t_{n-1}}(x_k))_{1\leq i\leq n-1, 1\leq k\leq l},$ and a function 
$$
f(x^{(1)}_1,\ldots,x^{(l)}_1,\ldots,x^{(1)}_{n-1},\ldots,x^{(l)}_{n-1})=\prod^{n-1}_{i=1}\prod^l_{k=1} g^{(k)}_{i,n}(x^{(k)}_i).
$$ 
One has
$$
\begin{aligned}
\mathbb{E} & \prod_{1\leq i<j\leq n}\prod^l_{k=1} g^{(k)}_{i,j}(\psi_{t_i,t_j}(x_k))\\
=\mathbb{E} & \prod_{1\leq i<j\leq n-1}\prod^l_{k=1} g^{(k)}_{i,j}(\psi_{t_i,t_j}(x_k)) \mathbb{E}\Big(f(\psi_{t_{n-1},t_n}(\xi))/\mathcal{F}^\psi_{t_{n-1}}\Big)\\
=\mathbb{E}& \Big(T^{(l(n-1))}_{t_n-t_{n-1}}f(\xi)\prod_{1\leq i<j\leq n-1}\prod^l_{k=1} g^{(k)}_{i,j}(\psi_{t_i,t_j}(x_k))\Big).
\end{aligned}
$$
Hence, distribution of $\{\psi_{t_i,t_j}(x_k):1\leq i<j\leq n, 1\leq k\leq l\}$ is expressed in terms of the distribution of $\{\psi_{t_i,t_j}(x_k):1\leq i<j\leq n-1, 1\leq k\leq l\}.$ Finally, by {\bf SF3} the distribution of $(\psi_{t_1,t_2}(x_1),\ldots,\psi_{t_1,t_2}(x_l))$ is given by $P^{(l)}_{t_2-t_1}((x_1,\ldots,x_l),\cdot).$

From the condition {\bf SF3} one can easily deduce two usual properties: independence and stationarity of increments:

{\bf SF3'}  the distribution of the element $(\psi_{s,t}(x_1),\ldots,\psi_{s,t}(x_n))$ is $P^{(n)}_{t-s}((x_1,\ldots,x_n),\ldots);$

{\bf SF3''} for any $t_0\leq t_1\leq \ldots\leq t_n$ $\sigma-$fields $\sigma(\{\psi_{t_{i-1},t_i}(x):x\in \mathbb{R}\}),$ $1\leq i\leq n$ are independent.

In fact, in \cite{LJR} these two properties are intended instead of {\bf SF3}.

We do not know if it is possible to recover finite-dimensional distributions of the stochastic flow $\{\psi_{s,t}(x):-\infty<s\leq t<\infty, x\in \mathbb{R}\}$ from properties {\bf SF1, SF2, SF3', SF3''} (as it is stated in \cite[Rem. 1.4]{LJR}). At least, the natural approach we used above is inapplicable because conditions {\bf SF3', SF3''} are not enough to use the Fubini theorem for the mapping 
$$
C:M\times M\to M, \ C(f,g)=f\circ g,
$$
where $M$ is the space of all measurable function with cylindrical $\sigma-$field. On the contrary, the condition {\bf SF3} postulates the possibility to use the Fubini theorem. Analogous definition of the distribution of a stochastic flow was proposed in \cite{Darling}.

In view of this discussion it is natural to ask if one can construct  two stochastic flows $\{\psi_{s,t}(x): -\infty<s\leq <\infty, x\in \mathbb{R}\}$  and 
$\{\tilde{\psi}_{s,t}(x): -\infty<s\leq <\infty, x\in \mathbb{R}\}$ 
that statisfy conditions {\bf SF1, SF2, SF3', SF3''}, have same transition probabilities $\{P^{(n)}:n\geq 1\}$ but different finite-dimensional distributions. 

Below we give such example in discrete time -- there exist two stochastic flows (on $[0,1]$) 
$$
\{\psi_{s,t}(x): s,t\in\{0,2\}, s\leq t,x\in [0,1]\} \mbox{ and } \{\tilde{\psi}_{s,t}(x): s,t\in\{0,2\}, s\leq t,x\in [0,1]\}
$$ 
that statisfy conditions {\bf SF1, SF2, SF3', SF3''}, have the same transition probabilities $\{P^{(n)}:n\geq 1\}$ but different finite-dimensional distributions.

Consider a probability space $\Omega=[0,1]^2$ equipped with a Borel $\sigma-$filed and a Lebesgue measure $\mathbb{P}.$ Consider following random mappings of $[0,1]:$
$$
\psi_{0,1}(\omega_1,\omega_2;x)=\tilde{\psi}_{0,1}(\omega_1,\omega_2;x)=\omega_1,
$$
$$
\psi_{1,2}(\omega_1,\omega_2;x)=\begin{cases}
\omega_2, \ x\ne \omega_1 \\
\omega_1, \ x=\omega_1
\end{cases},\ 
\tilde{\psi}_{1,2}(\omega_1,\omega_2;x)=\omega_2,
$$
$$
\tilde{\psi}_{s,s}(x)=\psi_{s,s}(x)=x, \ s\in\{0,1,2\},
$$
and their compositions
$$
\psi_{0,2}(\omega_1,\omega_2;x)=\psi_{1,2}(\omega_1,\omega_2;\psi_{0,1}(\omega_1,\omega_2;x))=\omega_1,
$$
$$
\tilde{\psi}_{0,2}(\omega_1,\omega_2;x)=\tilde{\psi}_{1,2}(\omega_1,\omega_2;\tilde{\psi}_{0,1}(\omega_1,\omega_2;x))=\omega_2.
$$
Realizations of these two flows are illustrated in figure 7.

\begin{figure}[!h]
	\centering
	\caption{Trajectories of flows $\psi$ and $\tilde{\psi}.$}
	\includegraphics[width=8cm]{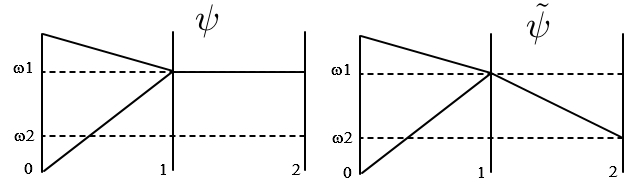}
\end{figure}

Obviously, both families are measurable strong flows.  $\sigma$-fields generated by $\tilde{\psi}_{0,1}=\psi_{0,1}$ and $\tilde{\psi}_{1,2}$ are $\{B\times[0,1], B\in \mathcal{B}([0,1])\}$ and $\{[0,1]\times B, B\in \mathcal{B}([0,1])\}$ respectively. Hence they are independent. Further, each set from  $\sigma(\psi_{1,2})$ differs from the set $[0,1]\times B$ by the set of a Lebesgue measure zero. Hence $\sigma(\psi_{0,1})$ and $\sigma(\psi_{1,2})$ are independent also. Finite-dimensional distributions of all mappings $\psi_{0,1},$ $\psi_{1,2},$ $\psi_{0,2},$ $\tilde{\psi}_{0,1},$ $\tilde{\psi}_{1,2},$ $\tilde{\psi}_{0,2},$  are the same and given by 
$$
\mathbb{P}(\psi(x_1)=\ldots=\psi(x_n)\in (a,b))=b-a, \ 0\leq a\leq b\leq 1,
$$
i.e. each mapping sends all segment $[0,1]$ into one (uniformly distributed) random point. 	So, conditions {\bf SF1,SF2,SF3',SF3''} are satisfied for both flows. However, flows are different (see figure 7). Respectively, finite-dimensional distributions of these flows are different: 
$\tilde{\psi}_{0,1}$ and $\tilde{\psi}_{0,2}$ are independent while $\psi_{0,1}=\psi_{0,2}.$

\end{document}